\documentclass{amsart}

\usepackage{amssymb}
\usepackage{amsmath}
\newcommand{\erase}[1]{}

\newtheorem{theorem}{Theorem}[section]
\newtheorem{proposition}[theorem]{Proposition}

\newtheorem{_remark}[theorem]{\it Remark}
\newenvironment{remark}{\begin{_remark}\rm}{\end{_remark}}
\newtheorem{_definition}[theorem]{\bf Definition}
\newenvironment{definition}{\begin{_definition}\rm}{\end{_definition}}

\numberwithin{equation}{section}
\numberwithin{table}{section}
\numberwithin{figure}{section}

\newcommand{\F}{\mathord{\mathbb F}}
\renewcommand{\P}{\mathord{\mathbb  P}}

\newcommand{\Z}{\mathord{\mathbb  Z}}

\newcommand{\CCC}{\mathord{\mathcal C}}
\newcommand{\DDD}{\mathord{\mathcal D}}
\newcommand{\EEE}{\mathord{\mathcal E}}

\newcommand{\KKK}{\mathord{\mathcal K}}
\newcommand{\NNN}{\mathord{\mathcal N}}
\newcommand{\PPP}{\mathord{\mathcal P}}
\newcommand{\QQQ}{\mathord{\mathcal Q}}
\newcommand{\RRR}{\mathord{\mathcal R}}
\newcommand{\SSS}{\mathord{\mathcal S}}
\newcommand{\VVV}{\mathord{\mathcal V}}

\newcommand{\graphG}{G}

\newcommand{\shortset}[2]{\{ {#1} \,|\, {#2}   \}}

\newcommand{\isom}{\mathbin{\,\raise -.6pt\rlap{$\to$}\raise 3.5pt \hbox{\hskip .3pt$\mathord{\sim}$}\,}}
\newcommand{\set}[2]{\{\; {#1} \; \mid \; {#2} \;  \}}
\newcommand{\wt}{\widetilde}

\newcommand{\sprime}{\sp\prime}
\newcommand{\spprime}{\sp{\prime\prime}}
\newcommand{\inv}{\sp{-1}}
\newcommand{\PGU}{\mathord{\mathrm {PGU}}}
\newcommand{\PSU}{\mathord{\mathrm {PSU}}}
\newcommand{\PGL}{\mathord{\mathrm {PGL}}}
\newcommand{\PSL}{\mathord{\mathrm {PSL}}}
\newcommand{\PGO}{\mathord{\mathrm {PGO}}}

\newcommand{\Aut}{\operatorname{\mathrm {Aut}}\nolimits}

\newcommand{\PK}{\PPP\KKK}

\newcommand{\vp}{\mathord{\bf p}}
\newcommand{\vx}{\mathord{\bf x}}
\newcommand{\vP}{\mathord{\bf P}}

\newcommand{\Plist}{\mathord{\tt P}}
\newcommand{\Qlist}{\mathord{\tt Q}}

\newcommand{\Slist}{\mathord{\tt S}}
\newcommand{\SQlist}{\mathord{\tt SQ}}
\newcommand{\EQlist}{\mathord{\tt EQ}}
\newcommand{\Dlist}{\mathord{\tt D}}
\newcommand{\Klist}{\mathord{\tt K}}

\newcommand{\Tmat}{\mathord{\mathbb T}}
\newcommand{\stab}{\mathord{\rm stab}}
\newcommand{\Conj}{\mathord{\rm Conj}}

\newcommand{\AAAA}{\mathord{\mathfrak{A}}}
\newcommand{\DDDD}{\mathord{\mathfrak{D}}}
\newcommand{\SSSS}{\mathord{\mathfrak{S}}}

\newcommand{\gen}[1]{\langle #1\rangle}

\begin{document}

\title[Hoffman-Singleton, Higman-Sims,  McLaughlin, and Hermitian curve]
{The graphs of Hoffman-Singleton, Higman-Sims,  and McLaughlin,  and the Hermitian curve\\
 of degree $6$ in characteristic $5$}

\author{Ichiro Shimada}
\address{
Department of Mathematics, 
Graduate School of Science, 
Hiroshima University,
1-3-1 Kagamiyama, 
Higashi-Hiroshima, 
739-8526 JAPAN
}
\email{shimada@math.sci.hiroshima-u.ac.jp}

\thanks{Partially supported by
JSPS Grants-in-Aid for Scientific Research (C) No.25400042}

\begin{abstract}
We construct the graphs of Hoffman-Singleton, Higman-Sims,  and   McLaughlin
 from certain  relations
on the set of non-singular conics totally tangent to the Hermitian curve of degree $6$
in characteristic $5$. 
We then interpret this geometric construction 
in terms of the subgroup structure of the automorphism group of this Hermitian curve. 
\end{abstract}


%05   (1940-now) Combinatorics [For finite fields, see 11Txx]
%05C   (1973-now) Graph theory [For applications of graphs, see 68R10, 81Q30, 81T15, 82B20, 82C20, 90C35, 92E10, 94C15]
%05C62   (2000-now) Graph representations (geometric and intersection representations, etc.) 
%51   (1980-now) Geometry [For algebraic geometry, see 14-XX]
%51E   (1980-now) Finite geometry and special incidence structures
%51E20   (1980-now) Combinatorial structures in finite projective spaces [See also 05Bxx]
%05   (1940-now) Combinatorics [For finite fields, see 11Txx]
%05B   (1973-now) Designs and configurations [For applications of design theory, see 94C30]
%05B25   (1973-now) Finite geometries [See also 51D20, 51Exx]
%20B25  	Finite automorphism groups of algebraic, geometric, or combinatorial structures [See also 05Bxx, 12F10, 20G40, 20H30, 51-XX]
%05C25   (1973-now) Graphs and abstract algebra (groups, rings, fields, etc.) [See also 20F65]

\maketitle

% main text starts here

\section{Introduction}\label{sec:Intro}
The Hoffman-Singleton graph~\cite{MR0140437}
is the unique strongly regular graph of parameters
$(v, k, \lambda, \mu)=(50, 7, 0, 1)$.
Its automorphism group
contains  $\PSU_{3}(\F_{25})$  as a subgroup of index $2$.
The Higman-Sims graph~\cite{MR0227269}
is the unique strongly regular graph of parameters
$(v, k, \lambda, \mu)=(100, 22, 0, 6)$. See~\cite{MR0304232} for the uniqueness.
Its automorphism group
contains the Higman-Sims group  as a subgroup of index $2$.
The McLaughlin graph~\cite{MR0242941} is the unique strongly regular graph of parameters
$(v, k, \lambda, \mu)=(275, 112,30,56)$.
See~\cite{MR0384597} for the uniqueness.
Its automorphism group
contains the McLaughlin group  as a subgroup of index $2$.
Many constructions of these  beautiful graphs
are known (see, for example,~\cite{MR782310}).
\par
\smallskip
These three graphs are closely related.
The Higman-Sims graph has been  constructed from the set of $15$-cocliques in the Hoffman-Singleton graph
(see~Hafner~\cite{MR2114181}).
Recently, the McLaughlin graph has been constructed from  the Hoffman-Singleton graph 
by Inoue~\cite{MR2917927}.
\par
\smallskip
On the other hand,
looking at  the  automorphism group,
one  naturally expects a relation of the Hoffman-Singleton graph
with the classical unital in  $\P^2(\F_{25})$.
In fact, Benson and Losey~\cite{MR0281658} constructed the Hoffman-Singleton graph 
by means  of the  geometry  of $\P^2(\F_{25})$ equipped with a Hermitian polarity.
They constructed a bijection between the set of  claws 
in the graph and the set of polar triangles 
on the plane compatible with the natural action of $\PSU_{3}(\F_{25})$.
\par
\smallskip
In this paper,
we give a unified  geometric construction (Theorems~\ref{thm:mainHfSg}~and~\ref{thm:mainHgSm}) of 
the Hoffman-Singleton graph and 
the Higman-Sims graph    by means of
 non-singular conics totally tangent to the Hermitian curve $\Gamma_5$ of degree $6$
in characteristic $5$. 
Using this result, we recast Inoue's construction~\cite{MR2917927} of the McLaughlin graph
in a simpler form (Theorem~\ref{thm:mainMcL}).
We then translate this construction to a group-theoretic construction
in terms of the subgroup structure of the automorphism group $\PGU_3(\F_{25})$ of the curve $\Gamma_5$
(Theorems~\ref{thm:S5},~\ref{thm:T7} and~\ref{thm:A7}).
In fact, it turns out at the final stage that we can construct the  graphs  
without mentioning any geometry (Theorem~\ref{thm:purelygroup}).
\subsection{Geometric construction}\label{subsec:geomconst}
In order to emphasize the algebro-geometric character of our construction,
we work, not over $\F_{25}$, but  over an algebraically closed field $k$ of characteristic $5$.
A projective plane curve $\Gamma$ of degree $6$ is said to be a \emph{$k$-Hermitian curve}
if $\Gamma$ is projectively isomorphic to the  Fermat  curve
$$
{\Gamma_5}\;\;: \;\;x^6+y^6+z^6=0
$$
of degree $6$.
Several characterizations of $k$-Hermitian curves are known;
for example, see~\cite{MR1092144} for a characterization
by  reflexivity.
Geometric  properties of $k$-Hermitian curves that will be used in the following can be found in~\cite{MR0200782},~\cite{MR0213949},
or~\cite[Chap.~23]{MR1363259}.
%(See also Section~\ref{sec:crossratio} of this paper.)
Let $\Gamma$ be a $k$-Hermitian curve.
Its automorphism group 
$$
\Aut(\Gamma)=\set{g\in \PGL_3(k)}{g(\Gamma)=\Gamma}
$$
is conjugate to $\Aut({\Gamma_5})=\PGU_3(\F_{25})$ in $\Aut(\P^2)=\PGL_3(k)$.
In particular, its order is $378000$,
and it contains a subgroup of index $3$ isomorphic to  the simple group $\PSU_3(\F_{25})$.
Let $P$ be a point of $\Gamma$.
Then the tangent line $l_P$ to $\Gamma$ at $P$ intersects $\Gamma$ at $P$ with intersection multiplicity $\ge 5$.
When  $\Gamma=\Gamma_5$,
the line $l_P$  intersects $\Gamma$ at $P$ with intersection multiplicity $6$
if and only if $P$ is an $\F_{25}$-rational point.
Combining this fact with the result of~\cite{MR847097} and~\cite{MR1373551},
we see that 
a point $P$ of $\Gamma$ is a Weierstrass point of $\Gamma$ if and only if
$l_P$  intersects $\Gamma$ at $P$ with intersection multiplicity $6$.
Let ${\PPP}$ denote the set of Weierstrass points of $\Gamma$.
Then we have $|{\PPP}|=126$.
The group  $\Aut(\Gamma)$ acts on ${\PPP}$ doubly transitively.
\begin{definition}\label{def:secant}
A line $L$ of $\P^2$ is a \emph{special secant line of $\Gamma$} 
if $L$ passes through two distinct points of ${\PPP}$.
\end{definition}
Let $\SSS$ denote the set of special secant lines of $\Gamma$.
Then we have 
$|\SSS|=525$.
The group  $\Aut(\Gamma)$ acts on $\SSS$ transitively.
If $L\in \SSS$,
then we have $|L\cap \Gamma|=6$ and 
$$
L\cap \Gamma\subset {\PPP}.
$$
\par
\smallskip
The incidence structure on the set of Weierstrass points of
a Hermitian curve (over an arbitrary finite field) induced by special secant lines 
has been studied by many authors.
See~\cite{MR2071863} for these works.
\begin{definition}\label{def:ttconics}
A  non-singular conic $Q$ on $\P^2$ is said to be \emph{totally tangent to $\Gamma$}
if $Q$ intersects $\Gamma$ at six distinct  points with intersection multiplicity $2$.
\end{definition}
\par
\smallskip
Let $\QQQ$ denote the set of non-singular conics totally tangent to $\Gamma$.
Then  we have
$|\QQQ|=3150$.
The group  $\Aut(\Gamma)$ acts on $\QQQ$ transitively.
For each $Q\in \QQQ$, 
we have
$$
Q\cap \Gamma\subset {\PPP}.
$$
A special secant line $L$ of $\Gamma$ is said to be a \emph{special secant line of $Q \in \QQQ$}
if $L$ passes through  two distinct points of $Q\cap \Gamma$.
We denote by $\SSS(Q)$ the set of special secant lines of $Q$.
Since $|Q\cap \Gamma|=6$,  we obviously have $|\SSS(Q)|=15$.
\par
\smallskip
Non-singular conics totally tangent to a Hermitian curve  were 
investigated by B.~Segre~\cite[n.~81]{MR0213949}.
See also~\cite{MR3092762}
for a simple proof of a higher dimensional analogue of Segre's results.
\par
\smallskip
A \emph{triangular graph} $T(m)$ is defined to be the graph 
whose set of vertices is the set of unordered pairs of distinct elements of 
$\{1,2, \dots, m\}$
and whose set of edges is the set of pairs $\{\{i,j\}, \{i\sprime, j\sprime\}\}$ such that
$\{i,j\}\cap \{i\sprime, j\sprime\}$ is non-empty (see~\cite{MR0406822}).
It is easy to see 
 that $T(m)$ is a strongly regular graph of parameters $(v, k, \lambda, \mu)=(\,m(m-1)/2,\, 2(m-2),\, m-2, \,4\,)$.
\par
\smallskip
Our construction  proceeds as follows.
\begin{proposition}\label{prop:G}
Let $\graphG$ be the graph whose set of vertices  is 
$\QQQ$ and whose set of edges is the set of pairs 
$\{Q, Q\sprime\}$  of distinct conics in $\QQQ$
such that $Q$ and $Q\sprime$ intersect transversely {\rm(}that is, $|Q\cap Q\sprime|=4${\rm)} and 
$|\SSS(Q)\cap \SSS(Q\sprime)|=3$.
Then $\graphG$ has exactly $150$ connected components,
and each connected component $D$ is isomorphic to
the triangular graph $T(7)$.
\end{proposition}
Let $\DDD$ denote the set of connected components 
of the graph $\graphG$.
Each $D\in \DDD$ is a collection of $21$ conics in $\QQQ$
that satisfy the following property:
\begin{proposition}\label{prop:geomD}
Let $D\in \DDD$ be a connected component of $\graphG$.
Then we have
$$
|Q\cap Q\sprime\cap \Gamma|=0
$$
for any distinct conics $Q, Q\sprime$ in $D$.
%Each connected component  $D\in DDD$
%satisfies
%$$
%\bigcup_{Q\in D} (Q\cap \Gamma)={\PPP},
%$$
%that is, each pair $Q, Q\sprime$ of elements of $D$ satisfies $|Q\cap Q\sprime\cap \Gamma|=0$.
Since $|D|\times 6=|{\PPP}|$, 
each connected component $D$ of $\graphG$ gives rise to a decomposition of ${\PPP}$ into 
a disjoint union of  $21$  sets $Q\cap \Gamma$ of six points, where  $Q$ runs through $D$.
\end{proposition}
Using $\DDD$ as the set of vertices,
we construct two graphs $H$ and $H\sprime$
that contain the Hoffman-Singleton graph and the  Higman-Sims graph,
respectively.
\begin{proposition}\label{prop:typet}
Suppose that  $Q\in \QQQ$ and $D\sprime\in \DDD$ satisfy $Q\notin D\sprime$.
Then  one of the following holds:
\begin{eqnarray*}
{\rm (\alpha)} && 
|Q\cap Q\sprime\cap\Gamma|=\begin{cases}
2 & \textrm{for $3$ conics $Q\sprime\in D\sprime$,} \\
0 & \textrm{for $18$  conics $Q\sprime\in D\sprime$.}
\end{cases} \\
{\rm (\beta)} &&
|Q\cap Q\sprime\cap\Gamma|=\begin{cases}
2 & \textrm{for $1$ conic $Q\sprime\in D\sprime$,} \\
1 & \textrm{for $4$ conics $Q\sprime\in D\sprime$,} \\
0 & \textrm{for $16$  conics $Q\sprime\in D\sprime$.}
\end{cases} \\
{\rm (\gamma)} &&
|Q\cap Q\sprime\cap\Gamma|=\begin{cases}
1 & \textrm{for $6$ conics $Q\sprime\in D\sprime$,} \\
0 & \textrm{for $15$  conics $Q\sprime\in D\sprime$.}
\end{cases}
\end{eqnarray*}
\end{proposition}
%
%\begin{definition}
For $Q\in \QQQ$ and $D\sprime\in \DDD$ satisfying  $Q\notin D\sprime$, we define
$t(Q, D\sprime)$ to be $\alpha, \beta$ or $\gamma$ according to the cases in  Proposition~\ref{prop:typet}.
%\end{definition}
%
%
\begin{proposition}\label{prop:typeT}
Suppose that  $D, D\sprime\in \DDD$ are distinct,
and hence disjoint as subsets of $\QQQ$.
Then  one of the following holds:
\begin{eqnarray*}
{\rm (\beta^{21})} 
&&
t(Q, D\sprime)=\beta \quad \textrm{for all $Q\in D$}.
\\
{\rm (\gamma^{21})} 
&&
t(Q, D\sprime)=\gamma \quad \textrm{for all $Q\in D$}.
\\
{\rm (\alpha^{15}\gamma^{6})} 
&&
t(Q, D\sprime)=\begin{cases}
\alpha & \textrm{for $15$ conics $Q\in D$,} \\
\gamma & \textrm{for $6$  conics $Q\in D$.}
\end{cases}
\\
{\rm (\alpha^{3}\gamma^{18})} 
&&
t(Q, D\sprime)=\begin{cases}
\alpha & \textrm{for $3$ conics $Q\in D$,} \\
\gamma & \textrm{for $18$  conics $Q\in D$.}
\end{cases}
\end{eqnarray*}
\end{proposition}
%
%\begin{definition}
For distinct $D, D\sprime\in \DDD$, we define
$T(D, D\sprime)$ to be $\beta^{21}$, $\gamma^{21}$, $\alpha^{15}\gamma^{6}$ or $\alpha^{3}\gamma^{18}$
  according to the cases in  Proposition~\ref{prop:typeT}.
%\end{definition}
%
\begin{proposition}\label{prop:symtypeT}
For distinct $D, D\sprime\in \DDD$, we have $T(D, D\sprime)=T(D\sprime, D)$.
For a fixed $D\in \DDD$, the number of $D\sprime\in \DDD$ such that $T(D, D\sprime)=\tau$ is
$$
\begin{cases}
30 & \textrm{if $\tau=\beta^{21}$}, \\
42 & \textrm{if $\tau=\gamma^{21}$}, \\
7 & \textrm{if $\tau=\alpha^{15}\gamma^6$}, \\
70 & \textrm{if $\tau=\alpha^{3}\gamma^{18}$}. 
\end{cases}
$$
\end{proposition}
Our main results are as follows.
\begin{theorem}\label{thm:mainHfSg}
Let $H$ be the graph whose set of vertices is $\DDD$,
and whose set of edges  is the set of pairs $\{D, D\sprime\}$
such that $D\ne D\sprime$ and $T(D, D\sprime)=\alpha^{15}\gamma^{6}$.
Then $H$ has exactly three connected components,
and each connected component is the Hoffman-Singleton graph.
\end{theorem}
We denote by $\CCC_1, \CCC_2, \CCC_3$ the set of vertices of the connected components of $H$.
We have $|\CCC_1|=|\CCC_2|=|\CCC_3|=50$ and $\CCC_1\cup  \CCC_2\cup \CCC_3=\DDD$.
\begin{proposition}\label{prop:TDD}
If $D$ and  $D\sprime$ are in the same connected component of $H$,
then $T(D, D\sprime)$ is either $\gamma^{21}$ or  $\alpha^{15}\gamma^{6}$.
If $D$ and  $D\sprime$ are in different connected components of $H$,
then $T(D, D\sprime)$ is either $\beta^{21}$ or  $\alpha^{3}\gamma^{18}$.
\end{proposition}
\begin{theorem}\label{thm:mainHgSm}
Let $H\sprime$ be the graph whose set of vertices is $\DDD$,
and whose set of edges  is the set of pairs $\{D, D\sprime\}$
such that $D\ne D\sprime$ and $T(D, D\sprime)$ is either $\beta^{21}$ or  $\alpha^{15}\gamma^{6}$.
Then $H\sprime$ is a connected regular graph of valency $37$.
For any $i$ and $j$ with $i\ne j$, the restriction $H\sprime|(\CCC_i\cup \CCC_j)$ of $H\sprime$ to $\CCC_i\cup \CCC_j$ is
the Higman-Sims graph.
\end{theorem}
%
%
%We recall the construction of the Higman-Sims graph by $15$-cocliques in the Hoffman-Singleton graph.
The number of $15$-cocliques
in the Hoffman-Singleton graph is  $100$.
Connecting two distinct $15$-cocliques when  they have $0$ or $8$ common vertices,
we obtain the Higman-Sims graph.
Starting from  Robertson's  pentagon-pentagram construction~\cite{RobertsonPhDthesis}
(see also~\cite{MR2002216})
of the Hoffman-Singleton graph and using this $15$-coclique method,
Hafner~\cite{MR2114181} gave an elementary construction of  the Higman-Sims graph.
\par
\smallskip
Our construction 
is related to this construction
via the following:
\begin{proposition}\label{prop:coqliques}
Suppose that $i\ne j\ne k \ne i$.
Then the map
$$g_k\;\;:\;\; D\in \CCC_i \cup \CCC_j  \;\; \mapsto\;\; \set{D\sprime\in \CCC_k}{T(D, D\sprime)=\beta^{21}}
$$
induces a bijection from  $\CCC_i \cup \CCC_j$  to the set of $15$-cocliques 
in the Hoffman-Singleton graph $H|\CCC_k$.
For distinct $D, D\sprime\in \CCC_i \cup \CCC_j$, we have
$$
|g_k(D) \cap g_k(D\sprime)|=
\begin{cases}
0 & \textrm{if $T(D, D\sprime)=\alpha^{15}\gamma^6$},\\
3 & \textrm{if $T(D, D\sprime)=\alpha^{3}\gamma^{18}$},\\
5 & \textrm{if $T(D, D\sprime)=\gamma^{21}$},\\
8 & \textrm{if $T(D, D\sprime)=\beta^{21}$}.\\
\end{cases}
$$
\end{proposition}
Let $\EEE_1$ denote the set of edges of the Hoffman-Singleton graph $H|\CCC_1$; that is, 
$$
\EEE_1:=\set{\{D_1, D_2\} }{D_1, D_2 \in \CCC_1, \;\; T(D_1, D_2)=\alpha^{15}\gamma^6}.
$$
We define a symmetric relation $\sim$ on $\EEE_1$ by 
$\{D_1, D_2\} \sim \{D_1\sprime, D_2\sprime\}$
if and only if 
$\{D_1, D_2\}$ and $ \{D_1\sprime, D_2\sprime\}$ are disjoint and
there exists an edge $\{D_1\spprime, D_2\spprime\}\in \EEE_1$
that has a common vertex  with each of the edges $\{D_1, D_2\}$ and $ \{D_1\sprime, D_2\sprime\}$.
By Haemers~\cite{MR627493},
the graph  $(\EEE_1, \sim)$ is a  strongly regular graph of parameters 
$(v, k, \lambda, \mu)=(175, 72, 20, 36)$.
\par
\smallskip
Combining our results with the construction 
of the McLaughlin graph 
due to Inoue~\cite{MR2917927}, 
we obtain the following:
\begin{theorem}\label{thm:mainMcL}
Let $H\spprime$ be the graph whose set of vertices
is $\EEE_1\cup \CCC_2\cup \CCC_3$, and 
whose set of edges consists of 
\begin{itemize}
\item 
$\{E, E\sprime\}$, where $E, E\sprime\in \EEE_1$ are distinct and satisfy 
 $E \sim E\sprime$,
\item 
$\{E, D\}$, where $E=\{D_1, D_2\}\in \EEE_1$, $D\in  \CCC_2\cup \CCC_3$,
and both of  $T(D_1, D)$ and  $T(D_2, D)$ are $\alpha^3\gamma^{18}$, and 
\item $\{D, D\sprime\}$, where $D, D\sprime\in  \CCC_2\cup \CCC_3$ are distinct and satisfy 
and $T(D, D\sprime)=\alpha^{15}\gamma^6$ or $\alpha^3\gamma^{18}$.
\end{itemize}
Then $H\spprime$ is the McLaughlin graph.
\end{theorem}
%
%\begin{remark}
Since each vertex $D\in \DDD$ of $H$ and $H\sprime$ is  not a single point 
but a rather complicated geometric object (a collection of $21$ conics),
we can describe  
edges  of $H$ and $H\sprime$ by various geometric properties  other than $T(D, D\sprime)$.
Or conversely,
we can find interesting configurations of conics and lines 
from the graphs $H$ and $H\sprime$.
In Section~\ref{sec:other}, we present a few examples.
%\end{remark}
%
\begin{remark}
The graph $H\sprime$  of $150$ vertices has been constructed in~\cite{MR782310} and~\cite{MR2114181}.
\end{remark}
\subsection{Group-theoretic construction}\label{subsec:groupconst}
The automorphism group $\PGU_{3}(\F_{25})$
of the Fermat curve $\Gamma_5$ of degree $6$ in characteristic $5$
acts transitively on the sets $\QQQ$ and $\DDD$ of vertices of the graphs $\graphG$ and $H$ or $H\sprime$.
Using this fact, we can define the edges of $\graphG$, $H$ and $H\sprime$
by means of the structure of stabilizer subgroups in $\PGU_{3}(\F_{25})$.
For an element $x$ of a set $X$  on which $\PGU_{3}(\F_{25})$ acts,
we denote by $\stab(x)$  the stabilizer subgroup in $\PGU_{3}(\F_{25})$ of $x$.
By $\SSSS_m$ and $\AAAA_m$, we denote the symmetric group and the alternating group of degree $m$,
respectively.
\par
\smallskip
Let $Q$ be an element of $\QQQ$.
Then $\stab(Q)$ is isomorphic to $\PGL_2(\F_5)\cong\SSSS_5$
(see~\cite[n.~81]{MR0213949},~\cite{MR3092762} or   Proposition~\ref{prop:Q}).
A rather mysterious definition of the graph $\graphG$ in Proposition~\ref{prop:G}
can be  replaced by the following:
\begin{theorem}\label{thm:S5}
Let $Q$ and $Q\sprime$ be distinct elements of $\QQQ$.
Then $Q$ and $Q\sprime$ are adjacent in the graph $\graphG$ if and only if 
$\stab(Q) \cap \stab (Q\sprime)$ is isomorphic to $\AAAA_4$. 
Moreover,  $Q$ and $Q\sprime$ are in the same connected component of  $\graphG$ if and only if 
the subgroup $\gen{\stab(Q), \stab (Q\sprime)}$ of $\PGU_{3}(\F_{25})$ generated  by 
the union of $\stab(Q)$ and $\stab (Q\sprime)$
is isomorphic to $\AAAA_7$.
\end{theorem}
It is known that,
in the automorphism group of the Hoffmann-Singleton graph,
the stabilizer subgroup of a vertex is isomorphic to $\AAAA_7$
(see, for example, ~\cite[page 34]{MR827219}).
Proposition~\ref{prop:G} gives us a geometric interpretation of this isomorphism.
Note that the automorphism group $\Aut(T(m))$ of the triangular graph $T(m)$ is isomorphic to $\SSSS_m$
by definition.
\begin{theorem}\label{thm:T7}
For each element $D$ of $\DDD$,
the action of $\stab(D)$ on the triangular graph $D\cong T(7)$
identifies $\stab(D)$ with the subgroup $\AAAA_7$ of $\Aut(T(7))\cong \SSSS_7$.
\end{theorem}
In order to define the type $T(D, D\sprime)$
by means of the  structure of $\stab (D)\cong \AAAA_7$,
we define the following subgroups
of $\AAAA_7$.
See~\cite[pages 4 and 10]{MR827219}
for details.
Note that the full automorphism group of $\AAAA_7$ is $\SSSS_7$.
For a subgroup
$\Sigma$ of $\AAAA_7$, we put
$$
\Conj_{\AAAA}(\Sigma):=\shortset{g\inv \Sigma g}{g\in \AAAA_7},
\quad
\Conj_{\SSSS}(\Sigma):=\shortset{g\inv \Sigma g}{g\in \SSSS_7}.
$$
\par
\smallskip
(a) We put
$$
\Sigma_a:=\set{g\in \AAAA_7}{g(7)=7}.
$$
Then $\Sigma_a$ is isomorphic to $\AAAA_6$,
and is maximal in $\AAAA_7$.
Moreover we have 
$\Conj_{\AAAA}(\Sigma_a)=\Conj_{\SSSS}(\Sigma_a)$.
\par
\smallskip
(b)
We define a bijection $\rho: \P^2(\F_2)\isom \{1, \dots, 7\}$ by
$$
(a:b:c)\mapsto 4a+2b+c\quad(a, b\in \{0,1\}),
$$
and let $\rho\sprime: \P^2(\F_2)\isom \{1, \dots, 7\}$ be the composite of $\rho$ and the transposition $(67)$.
Then the action of $\PSL_3(\F_2)$  on $\P^2(\F_2)$
induces two faithful permutation  representations $\PSL_3(\F_2)\hookrightarrow \SSSS_7$
corresponding to  $\rho$ and $\rho\sprime$,
and the images $\Sigma_b$ and $\Sigma_b\sprime$ 
of these representations are  contained in $\AAAA_7$.
These subgroups $\Sigma_b$ and $\Sigma_b\sprime$ are of order $168$,
and  are maximal in $\AAAA_7$.
(Note that $\PSL_3(\F_2)\cong \PSL_2(\F_7)$.)
We also have 
$$
\Conj_{\SSSS}(\Sigma_b)=\Conj_{\SSSS}(\Sigma_b\sprime)=\Conj_{\AAAA}(\Sigma_b)\cup \Conj_{\AAAA}(\Sigma_b\sprime),
\quad
\Conj_{\AAAA}(\Sigma_b)\cap \Conj_{\AAAA}(\Sigma_b\sprime)=\emptyset.
$$
\par
\smallskip
(c)
We put
$$
\Sigma_c:=\set{g\in \AAAA_7}{\{g(5), g(6), g(7)\}=\{5,6,7\}}.
$$
Then $\Sigma_c$ is isomorphic to the group  $(\AAAA_4\times 3):2$ of order $72$,
and is maximal in $\AAAA_7$.
Moreover we have 
$\Conj_{\AAAA}(\Sigma_c)=\Conj_{\SSSS}(\Sigma_c)$.
\par
\smallskip
(d)
Because of the extra outer automorphism of $\SSSS_6$,
the group $\AAAA_6$ has  two  maximal subgroups
isomorphic to $\AAAA_5$ up to  inner automorphisms.
One is  a point stabilizer,
while  the other is  the stabilizer subgroup of a \emph{total} (a set of five synthemes containing all duads).
We fix a total
\begin{eqnarray*}
t_0&:=&
\{\{\{1, 2\}, \{3, 4\}, \{5, 6\}\}, \{\{1, 3\}, \{2, 5\}, \{4, 6\}\}, \{\{1, 4\}, \{2, 6\}, \{3, 5\}\}, \\
&&\{\{1, 5\}, \{2, 4\}, \{3, 6\}\}, \{\{1, 6\}, \{2, 3\}, \{4, 5\}\}\},
\end{eqnarray*}
and put
$$
\Sigma_d:=\set{g\in \AAAA_7}{g(7)=7,\;\; g(t_0)=t_0}.
$$
Then $\Sigma_d$ is isomorphic to $\AAAA_5$, and 
$\Conj_{\AAAA}(\Sigma_d)=\Conj_{\SSSS}(\Sigma_d)$ holds.
\par
\smallskip
Now we have the following:
\begin{theorem}\label{thm:A7}
Let $D$ and $D\sprime$ be distinct elements of $\DDD$.
We identify $\stab(D)$ with $\AAAA_7$ by Theorem~\ref{thm:T7}.
Then $T(D, D\sprime)$ is
$$
\begin{cases}
\beta^{21} & \textrm{if and only if $\stab(D) \cap \stab(D\sprime)$ is conjugate to $\Sigma_b$ or $\Sigma_b\sprime$}, \\
\gamma^{21} & \textrm{if and only if $\stab(D) \cap \stab(D\sprime)$ is conjugate to $\Sigma_d$}, \\
\alpha^{15}\gamma^6 & \textrm{if and only if $\stab(D) \cap \stab(D\sprime)$ is conjugate to $\Sigma_a$}, \\
\alpha^{3}\gamma^{18} & \textrm{if and only if $\stab(D) \cap \stab(D\sprime)$ is conjugate to $\Sigma_c$}.
\end{cases}
$$
%[60, cc21], [72, aa3cc18], [168, bb21], [360, aa15cc6], [2520, 0]
\end{theorem}
Note that the statement of Theorem~\ref{thm:A7} does not depend on the choice of the isomorphism
$\stab(D)\cong\AAAA_7$,
which is not unique up to 
conjugations by elements of $\AAAA_7$, 
but is unique up to conjugations by elements of $\SSSS_7$.
\par
\smallskip
Theorem~\ref{thm:A7} implies that $T(D, D\sprime)$ is determined 
simply by the order of the group $\stab(D) \cap \stab(D\sprime)$.
Combining this fact with Theorems~\ref{thm:mainHfSg},~\ref{thm:mainHgSm} and~\ref{thm:mainMcL},
we have obtained constructions of the three graphs in the title by the
subgroup structure of $\PGU_3(\F_{25})$.
(See also Theorem~\ref{thm:purelygroup}.)
\par
\smallskip
We fix $D\in \DDD$.
Let $\CCC_i$ be the connected component of $H$
containing $D$,
and let $\CCC_j$ and $\CCC_k$ be the other two connected components.
The set
$$
\NNN_D:=\set{D\sprime\in \DDD}{T(D, D\sprime)=\beta^{21}}
$$
of vertices that are adjacent to $D$ in $H\sprime$ but are not adjacent to $D$ in $H$
decomposes into the disjoint union of two subsets
$\NNN_D\cap \CCC_j$ and $\NNN_D\cap \CCC_k$.
\begin{proposition}\label{prop:twobb21}
Fixing $\stab(D)\cong \AAAA_7$ and 
interchanging $j$ and $k$ if necessary, we have the following;
for any $D\sprime\in \NNN_D$,  
\begin{eqnarray*}
D\sprime \in \NNN_D\cap \CCC_j\ &\Longleftrightarrow& 
\textrm{$\stab(D) \cap \stab(D\sprime)$ is conjugate to $\Sigma_b$},\\
D\sprime \in \NNN_D\cap \CCC_k\ &\Longleftrightarrow& 
\textrm{$\stab(D) \cap \stab(D\sprime)$ is conjugate to $\Sigma_b\sprime$}.
\end{eqnarray*}
\end{proposition}
\begin{remark}
We have 
$$
|\Conj_{\AAAA}(\Sigma)|=\begin{cases}
7 & \textrm{if $\Sigma=\Sigma_a$}, \\
15 & \textrm{if $\Sigma=\Sigma_b$ or $\Sigma_b\sprime$}, \\
35 & \textrm{if $\Sigma=\Sigma_c$}, \\
42 & \textrm{if $\Sigma=\Sigma_d$}. \\
\end{cases}
$$ 
Compare this result with Proposition~\ref{prop:symtypeT}. 
\end{remark}
\subsection{Plan of the paper}\label{subsec:plan}
Suppose that an adjacency matrix of a graph $\varGamma$ is given.
If the number of vertices 
is not very large, 
it is a simple task for a computer to calculate 
the adjacency matrices of connected components of $\varGamma$.
Moreover, when $\varGamma$ is connected,
it is also easy  to 
determine whether $\varGamma$ is strongly regular or not,
and in the case when  $\varGamma$ is strongly regular,
to compute its parameters $(v, k, \lambda, \mu)$.
Since the three graphs we are concerned with are 
strongly regular graphs  uniquely determined by the parameters,
Theorems~\ref{thm:mainHfSg},~\ref{thm:mainHgSm} and~\ref{thm:mainMcL}  
can be verified if we calculate the adjacency matrices
of $H$, $H\sprime$ and $H\spprime$.
These adjacency matrices are computed from the list $\QQQ$ of conics 
 by a simple geometry.
On the other hand,
it seems to be  a non-trivial computational task to calculate the list $\QQQ$.
Therefore, 
in Section~\ref{sec:crossratio},
we give a  method to calculate the list $\QQQ$.
This list and other auxiliary computational data are on the author's web page
\par
%\begin{equation}\label{eq:HScomp}
 \verb+     http://www.math.sci.hiroshima-u.ac.jp/~shimada/HSgraphs.html .+
\par
\noindent
In Sections~\ref{sec:proof}~and~\ref{sec:proof2},
we indicate how to prove our results by these computational data. 
In Section~\ref{sec:other},
we discuss other geometric methods of defining edges of the graphs $H$ and $H\sprime$.
\par
\smallskip
This work stems from the author's joint work~\cite{KKSchar5} with Professors T.~Katsura and S.~Kondo
on the geometry of a supersingular $K3$ surface in characteristic $5$.
The author expresses his gratitude to them 
for many  discussions and  comments.
He also thanks the referees 
for their many useful comments and suggestions on the first version of this paper.
\section{Other methods of defining edges of   $H$ and $H\sprime$}\label{sec:other}
In this section,
we present various ways of defining edges of  $H$ and $H\sprime$.
\subsection{Definition of the edges of $H$ by $6$-cliques}\label{subsec:6cliques}
By the  assertion $D\cong T(7)$ of Proposition~\ref{prop:G},
each $D\in \DDD$ contains exactly seven $6$-cliques of $\graphG$.
Let $\KKK$ denote the set of $6$-cliques in $\graphG$.
We have 	$|\KKK|= 1050$, and, for each $K\in \KKK$, we have 
$$
|\bigcup_{Q\in K} (Q\cap \Gamma)|=36.
$$
%for any $K\in \KKK$ by  definition.
%
%
\begin{proposition}\label{prop:uniqueK}
Let $K$ be a $6$-clique  of $\graphG$.
Then there exists a unique $6$-clique $K\sprime$ of $\graphG$ disjoint  from $K$ as a subset of $\QQQ$ such that
\begin{equation}\label{eq:cupK}
\bigcup_{Q\in K} (Q\cap \Gamma)=\bigcup_{Q\sprime\in K\sprime} (Q\sprime\cap \Gamma).
\end{equation}
\end{proposition}
We denote by $\PK$ the set of pairs $\{K, K\sprime\}$ of disjoint $6$-cliques of $\graphG$ satisfying~\eqref{eq:cupK}.
%Then $\KKK$ is decomposed into the union of $525$ pairs in $\PK$.
The edges  of the graph  $H$ can be defined as follows:
\begin{proposition}\label{prop:adjHbyK}
Two distinct vertices $D$ and $D\sprime$ of $H$
are adjacent  in $H$ if and only if 
there exists $\{K, K\sprime\} \in \PK$ 
such that $D$ contains $K$ and $D\sprime$ contains $K\sprime$.
\end{proposition}
In other words, the set  $\PK$ can be identified with the set of edges of $H$.
Each pair in $\PK$  has the following remarkable geometric 
property:
\begin{proposition}\label{prop:geomK}
Let $\{K, K\sprime\}$ be an element of $\PK$.
Then  any $Q\in K$ and any $Q\sprime\in K\sprime$ intersect at one point with intersection multiplicity $4$.
\end{proposition}
\subsection{Definition of the edges of $H$ by special secant lines}\label{subsec:SQ}
For $D\in \DDD$, we put
$$
\SSS_D:=\bigcup_{Q\in D} \SSS(Q).
$$
\begin{proposition}\label{prop:secD}
We have $|\SSS_D|=105$ for any $D\in \DDD$.
\end{proposition}
\begin{proposition}\label{prop:adjHsec}
Let $D$ and $D\sprime$ be distinct vertices of $H$. Then we have 
$$
|\SSS_D\cap \SSS_{D\sprime}|=\begin{cases}
45 & \textrm{if $D$ and $D\sprime$ are adjacent in $H$, }\\
15 & \textrm{if $D$ and $D\sprime$ are not adjacent }\\
& \textrm{\qquad\qquad but in the same connected component of $H$, }\\
21 & \textrm{if $D$ and $D\sprime$ are  in different   connected components of $H$.}
\end{cases}
$$
If  $D$ and $D\sprime$ are  in different  connected components of $H$,
then ${\PPP}$ is a disjoint union of $21$ sets of six points $L\cap\Gamma$,
where $L$ runs through $\SSS_D\cap \SSS_{D\sprime}$.
\end{proposition}
\subsection{Definition of the edges of $H\sprime$ by doubly tangential pairs of conics}\label{subsec:RRRTTT}
Suppose that $Q, Q\sprime\in \QQQ$ are distinct.
Since $Q$ and $Q\sprime$ are tangent at each point of $Q\cap Q\sprime\cap\Gamma$,
we have   $|Q\cap Q\sprime\cap\Gamma|\le 2$.
We say that a pair $\{Q, Q\sprime\}$ of conics
is \emph{doubly tangential} if $|Q\cap Q\sprime\cap\Gamma|=2$  holds.
It turns out that, if  $\{Q, Q\sprime\}$ is a doubly tangential pair,
then $|\SSS(Q)\cap \SSS(Q\sprime)|$ is either one or three~(see Table~\ref{table:IP}). 
%$1$ or $3$~(see Table~\ref{table:IP}).
For $Q\in \QQQ$, we put
$$
\RRR(Q) := \set{Q\sprime\in \QQQ}{|Q\cap Q\sprime\cap\Gamma|=2,\; |\SSS(Q)\cap \SSS(Q\sprime)|=1},
$$
and for $D\in \DDD$, we put $\RRR_D:=\bigcup_{Q\in D} \RRR(Q)$.
For each set  $\CCC_i$ of vertices of a connected component of $H$,
we put
$$
\wt{\CCC}_i:=\bigcup_{D\in \CCC_i} D\;\;\subset\;\;\QQQ.
$$
Then we have 
$|\wt{\CCC}_i|=1050$ and $\wt{\CCC}_1 \cup \wt{\CCC}_2\cup \wt{\CCC}_3=\QQQ$.
\begin{proposition}\label{prop:RRR}
For any $Q\in \QQQ$, 
we have $|\RRR(Q)|=45$ and $|\RRR(Q)\cap\wt{\CCC}_i|=15$ for $i=1, 2, 3$.
For any $D\in \DDD$, 
we have $|\RRR_D|=735$ and 
$$
|\RRR_D\cap\wt{\CCC}_i|=\begin{cases}
105 & \textrm{if $D\in \CCC_i$,}\\
315 & \textrm{if $D\notin \CCC_i$.}\\
\end{cases}
$$
\end{proposition}
%
%For a vertex $D\in \DDD$,
%we denote by $N\sprime (D)$ the set of vertices
%adjacent to $D$ in the graph $H\sprime$.
%
The sets $\RRR_D$ determine the edges of $H\sprime$ by the following:
\begin{proposition}\label{prop:RRR2}
Two distinct vertices   $D, D\sprime\in \DDD$ are adjacent in $H\sprime$
if and only if $\RRR_D$ and $D\sprime$ are not disjoint in $\QQQ$.
More precisely, we have
$$
|\RRR_D\cap D\sprime|=\begin{cases}
 21& \textrm{if $T(D, D\sprime)=\beta^{21}$, }\\
 15& \textrm{if $T(D, D\sprime)=\alpha^{15}\gamma^6$,}\\
 0& \textrm{if $T(D, D\sprime)=\gamma^{21}$ or $\alpha^{3}\gamma^{18}$}.
\end{cases}
$$
\end{proposition}
We can refine Proposition~\ref{prop:RRR2} for  edges of $H\sprime$ not contained in a connected component of $H$. 
For a vertex $D\in \DDD$,
we denote by $N\sprime (D)$ the set of vertices
adjacent to $D$ in the graph $H\sprime$.
If $D\in \CCC_i$ and $D\sprime\in N\sprime (D)\cap \CCC_j$ with $i\ne j$,
we have 
 $T(D, D\sprime)=\beta^{21}$,
 and hence, 
by definition of the type $\beta^{21}$ and Proposition~\ref{prop:symtypeT},
we can define   a unique  bijection $f_{D, D\sprime}: D\to D\sprime$ by requiring 
$|Q\cap f_{D, D\sprime}(Q)\cap \Gamma|=2$ for any $Q\in D$.
\begin{proposition}
If  $D\in \CCC_i$, $D\sprime\in \CCC_j$ with $i\ne j$ and $T(D, D\sprime)=\beta^{21}$, 
the conic $f_{D, D\sprime} (Q)$ belongs to $ \RRR(Q)$  for any $Q\in D$.
\end{proposition}
\begin{proposition}\label{prop:bb21}
For $D\in \CCC_i$ and $j\ne i$,  the map 
$(Q, D\sprime) \mapsto f_{D, D\sprime} (Q)$
induces a bijection
from  $D\times (N\sprime (D)\cap \CCC_j)$
to $\RRR_D\cap \wt{\CCC}_j$.
In other words, 
if $D\in \CCC_i$ and $j\ne i$, then $\RRR_D\cap \wt{\CCC}_j$
is a disjoint union of $D\sprime$, where $D\sprime$ runs through $N\sprime (D) \cap \CCC_j$.
\end{proposition}
\section{The list of totally tangent conics}\label{sec:crossratio}
In this section,
we give a method to calculate the list of conics totally tangent to a Hermitian curve 
in odd characteristic,
based on the results in~\cite[n.~81]{MR0213949} and~\cite{MR3092762}.
\par
\smallskip
Let $p$ be an odd prime, and $q$ a power of $p$.
We work in  characteristic $p$.
Let ${\Gamma_q}$ denote the Fermat curve
$$
x^{q+1} + y^{q+1} + z^{q+1}=0
$$
of degree $q+1$.
By~\cite{MR847097} and~\cite{MR1373551},
we see that, for a point $P$ of ${\Gamma_q}$,
the following  are equivalent;
\begin{itemize}
\item[(i)] the tangent line  of ${\Gamma_q}$ at $P$
intersects ${\Gamma_q}$ at $P$ with   multiplicity $q+1$, 
\item[(ii)] $P$ is an $\F_{q^2}$-rational point of ${\Gamma_q}$, and
\item[(iii)] $P$ is a Weierstrass point of $\Gamma_q$.
\end{itemize}
Let ${\PPP_q}$ denote the set of points of ${\Gamma_q}$ satisfying (i), (ii) and (iii).
By (ii), we have $|{\PPP_q}|=q^3+1$.
A non-singular conic $Q$ is said to be \emph{totally tangent to ${\Gamma_q}$} if $Q$ intersects ${\Gamma_q}$ at $q+1$ distinct  
points with intersection  multiplicity $2$.
Let ${\QQQ_q}$ denote the set of  non-singular conics totally tangent to ${\Gamma_q}$.
In~\cite[n.~81]{MR0213949} and~\cite{MR3092762},
the following was  proved:
\begin{proposition}\label{prop:Q}
{\rm (1)}
The automorphism group $\Aut({\Gamma_q})=\PGU_3(\F_{q^2})$ of ${\Gamma_q}$ acts 
on ${\QQQ_q}$ transitively, and 
the stabilizer subgroup %$\Stab(Q)$ 
of $Q\in {\QQQ_q}$ in $\Aut({\Gamma_q})$ 
is isomorphic to $\PGL_2(\F_q)$.
In particular, we have $|{\QQQ_q}|=q^2(q^3+1)$.

{\rm (2)}
For any $Q\in {\QQQ_q}$, we have $Q\cap {\Gamma_q}\subset {\PPP_q}$.
%Let $\P^1$ be the projective line defined over $\F_q$.
%Then  there exists an isomorphism $\P^1\tensor k \isom Q$
%defined over $k$
%that maps the set $\P^1(\F_q)$ of $\F_q$-rational points of $\P^1$ to $Q\cap {\Gamma_q}$
%bijectively and is  compatible with the action of $\PGL_2(\F_q)=\Aut(\P^1/\F_q)\cong \Stab (Q)$.

{\rm (3)}
Let $P_0, P_1, P_2$ be three distinct  points of ${\PPP_q}$.
If there exists a non-singular conic $Q$ that is tangent to ${\Gamma_q}$ at $P_0, P_1, P_2$,
then $Q\in {\QQQ_q}$.
\end{proposition}
A set $\{p_1, \dots, p_m\}$ of $m$ points in ${\PPP_q}$
with $m\ge 3$
is said to be a \emph{co-conical set of $m$ points}
if there exists $Q\in {\QQQ_q}$ such that
$$
\{p_1, \dots, p_m\} \subset Q\cap {\Gamma_q}.
$$
Then the set ${\QQQ_q}$
is identified with 
the set of co-conical sets of $q+1$ points via
$Q\mapsto Q\cap {\Gamma_q}$.
It is obvious that
any co-conical set of $q+1$ points $Q\cap {\Gamma_q}$
is a union of co-conical sets of $3$ points contained in $Q\cap {\Gamma_q}$, and 
the assertion (3) of Proposition~\ref{prop:Q} implies that
any co-conical set of $3$ points is contained in a unique 
co-conical set of $q+1$ points.
Therefore the following proposition,
which characterizes co-conical sets of $3$ points,  enables us to calculate ${\QQQ_q}$
efficiently.
%
%Therefore the list ${\QQQ_q}$
%can be computed  by the following:
%
\begin{proposition}\label{prop:crossratio}
Let $P_0, P_1, P_2$ be non-collinear three points of ${\PPP_q}$,
and let
$$
P_i=(a_i: b_i : c_i)
$$
be the homogeneous coordinates
of $P_i$ with $a_i, b_i, c_i\in \F_{q^2}$.
We put 
$$
\kappa_{ij}:=a_i\bar{a}_j + b_i \bar{b}_j +c_i\bar{c}_j\in \F_{q^2},
$$
where $\bar{x}:=x^q$ for $x\in \F_{q^2}$.
Then there exists a non-singular conic $Q$ that is tangent to  ${\Gamma_q}$ at $P_i$ 
for $i=0,1,2$
if and only if  
$\kappa_{12}\kappa_{23}\kappa_{31}$ is a non-zero element of $\F_q$.
%holds.
\end{proposition}
For the proof of Proposition~\ref{prop:crossratio}, 
we recall a classical result of elementary geometry.
Let $(L_0, L_1, L_2)$ be an ordered triple of non-concurrent lines on $\P^2$,
and let $P_0, P_1, P_2$ be  points of $\P^2$ such that
$P_i\in L_i$ and $P_i\notin L_j$ for $i\ne j$.
We denote by $V_i$ the intersection point of $L_j$ and $L_k$, where $i\ne j \ne k\ne i$.
We put
$$
\Delta:=(L_0, L_1, L_2\, |\, P_0, P_1, P_2).
$$
Let $T$ be the intersection point of the lines $V_1 P_1$ and $V_2 P_2$,
and let $R$ be the intersection point of $L_0$ and $V_0T$.
We denote by $\gamma(\Delta)$
the cross-ratio of the ordered four points 
$V_1, V_2, R, P_0$ on $L_0$;
that is,
if $z$ is an affine parameter of $L_0$, then 
%if the values of affine parameter  of $V_1, V_2, R, P_0$ on $L_0$ are
%$z_1, z_2, z_3, z_4$, then
$$
\gamma(\Delta):=\frac{(z(V_1)-z(R))(z(V_2)-z(P_0))}{(z(V_1)-z(P_0))(z(V_2)-z(R))}.
$$
It is easy to see that  there exists a non-singular conic 
that is tangent to $L_i$ at $P_i$ for $i=0,1,2$
if and only if $\gamma(\Delta)=1$.
\begin{proof}[Proof of Proposition~\ref{prop:crossratio}]
Let $L_i$ denote the tangent line of ${\Gamma_q}$ at $P_i$, which is defined by
$\bar{a}_i x + \bar{b}_i y +\bar{c}_i z=0$.
Since $P_0, P_1, P_2$ are not collinear,
the lines $L_0, L_1, L_2$ are not concurrent.
Since $L_i\cap {\Gamma_q}=\{P_i\}$, we can consider $\gamma(\Delta)$, where 
$\Delta=(L_0, L_1, L_2|P_0, P_1, P_2)$.
Since $\kappa_{ij}=\bar{\kappa}_{ji}$, 
it is enough to show that
\begin{equation}\label{eq:frackappas}
\gamma(\Delta)=\frac{\kappa_{12}\kappa_{23}\kappa_{31}}{\kappa_{21}\kappa_{32}\kappa_{13}}.
\end{equation}
Let $\vp_i$ denote the column vector ${}^t[a_i, b_i, c_i]$,
and we put
$\bar{\vp}_j:={}^t[\bar{a}_j, \bar{b}_j, \bar{c}_j]$.
We consider the unique  linear transformation $g$ of $\P^2$ that maps $L_0$ to $x=0$,
$L_1$ to $y=0$, $L_2$  to $z=0$,
and $P_1$ to $(1:0:1)$, $P_2$ to $(1:1:0)$.
Then $g$ maps $P_0$ to $[0: \gamma(\Delta):1]$.
Suppose that $g$ is given by the left multiplication of a $3\times 3$ matrix $M$.
%Since $L_i$ is defined by an equation ${}^t\vx \bar{\vp}_i=0$, 
Then there exist non-zero constants $\tau_0, \tau_1, \tau_2$ and $s, t, u$ such that
$$
M\vP=
\left[
\begin{array}{ccc}
0 & t & u\\
s\gamma(\Delta) & 0 & u \\
s & t & 0
\end{array}
\right],
\quad
{}^t M\inv \bar{\vP}=
\left[
\begin{array}{ccc}
\tau_0 & 0 & 0\\
0 & \tau_1 & 0 \\
0 & 0 & \tau_2
\end{array}
\right],
\;\;\textrm{where} 
\;\;
\vP:=[{\vp}_0, {\vp}_1, {\vp}_2].
$$
On the other hand, we have
$$
{}^t\vP\bar{\vP}=
\left[
\begin{array}{ccc}
0 & \kappa_{12} & \kappa_{13}\\
\kappa_{21} & 0 & \kappa_{23} \\
\kappa_{31} & \kappa_{32} & 0
\end{array}
\right].
$$
Combining these equations,  we obtain~\eqref{eq:frackappas}.
\end{proof}
\begin{remark}
By Chevalley-Warning theorem,
the non-singular conic $Q_1$ defined by $x^2+y^2+z^2=0$ has $q+1$ rational points over $\F_q$,
and it intersects $\Gamma_q$ at each of these $\F_q$-rational points with intersection multiplicity $2$.
Hence we can also make the list $\QQQ_q$ from $Q_1$ by the action of $\Aut(\Gamma_q)=\PGU_{3}(\F_{q^2})$ on $\P^2$.
\end{remark}
\section{Geometric construction}\label{sec:proof}
In this section, we
work in characteristic $5$.
By a  \emph{list}, we mean  an ordered  finite set.
We put
$$
\alpha:=\sqrt{2}\;\;\in\;\;  \F_{25}=\F_5(\alpha).
$$
With the help of Proposition~\ref{prop:crossratio}, we construct the following lists.
They are available  from the web page given at the end of Introduction. 
\begin{itemize}
\item The list $\Plist$ of the Weierstrass points of ${\Gamma_5}$;
that is, the list of $\F_{25}$-rational points of ${\Gamma_5}$.
\item The list $\Slist$ of  sets of  collinear six  points in $\Plist$, 
which  is regarded as the list $\SSS_5$ of special secant lines $L$  of ${\Gamma_5}$ by $L\mapsto L\cap {\Gamma_5}$.
\item The list $\Qlist$ of co-conical sets of  six  points in $\Plist$,
which  is regarded as the list $\QQQ_5$ of totally tangent conics $Q$  by $Q\mapsto Q\cap {\Gamma_5}$.
\end{itemize}
In the following, conics in ${\QQQ_5}$ are numbered as $Q_1, \dots, Q_{3150}$ according to the order of the list $\Qlist$.
For example,
the first member of $\Qlist$ is 
$$
\{\,(0:1:\pm 2),\;\; (1:0:\pm 2), \;\; (1:\pm 2: 0)\, \}, 
$$
and hence the conic $Q_1$ is defined by 
$\phi_1: x^2+y^2+z^2=0$.
From the lists $\Plist$, $\Slist$ and $\Qlist$, we obtain the following lists:
\begin{itemize}
\item The list $\SQlist=[\SSS(Q_1), \dots, \SSS(Q_{3150})]$  
of the sets $\SSS(Q_i)$ of  special secant lines of conics $Q_i$.
%Each member of $\SQlist$ is a subset of $\Slist$ of cardinality $15$.
\item The list $\EQlist=[\phi_1, \dots, \phi_{3150}]$  of the defining equations  $\phi_i$ of conics $Q_i$.
\end{itemize}
From the list $\Qlist$,
we compute a $3150\times 3150$  matrix $M_0$ 
whose $(i, j)$ entry is equal to  $|Q_i\cap Q_j\cap {\Gamma_5}|$.
From the list $\SQlist$,
we compute a $3150\times 3150$  matrix $M_1$ 
whose $(i, j)$ entry is equal to  $|\SSS(Q_i)\cap \SSS(Q_j)|$.
\par
\smallskip
Suppose that  a non-singular conic $C$ is defined by
an equation ${}^t\vx\, F_C\,  \vx=0$,
where $F_C$ is a $3\times 3$ symmetric matrix.
Then two non-singular conics $C$ and $C\sprime$ intersect at four distinct  points transversely
if and only if
the cubic polynomial 
$$
f:=\det (F_C+t F_{C\sprime})
$$
 of $t$ has no multiple roots.
From the list $\EQlist$, we calculate the discriminant of the polynomials $f$
for $Q_i, Q_j\in \QQQ_5$, and 
compute a $3150\times 3150$ matrix $M_2$
whose $(i, j)$ entry  is $1$ if $|Q_i\cap Q_j|=4$
and is $0$ if $i=j$ or $|Q_i\cap Q_j|<4$.
\par
\smallskip
From the matrices $M_1$ and $M_2$,
we  compute the adjacency matrix $A_{\graphG}$ of $\graphG$.
From the matrix $A_{\graphG}$,
we compute
the list $\Dlist$   of the connected components of $\graphG$.
It turns out that $\Dlist$ consists of $150$ members $D_1, \dots, D_{150}$,
and that each connected component
has $21$ vertices.
For example, the connected component $D_1$ of $\graphG$ containing $Q_1$
consists of the conics 
given in Table~\ref{table:D1},
 and their adjacency relation is
 given in Table~\ref{table:adjD1}, 
where  distinct conics $Q$ at the $i$th row and the $j$th column 
and  $Q\sprime$ at the $i\sprime$th row and the  $j\sprime$th column 
are adjacent if and only if $\{i, j\}\cap\{i\sprime, j\sprime\}\ne \emptyset$.
Thus an isomorphism $D_1\cong T(7)$ of graphs is established.
 Using the list $\Qlist$, we confirm
\begin{equation*}\label{eq:union126}
\bigcup_{Q\in D_1} (Q\cap {\Gamma_5}) ={\PPP_5}.
\end{equation*}
Since $\Aut({\Gamma_5})$ acts on ${\QQQ_5}$ transitively,
it acts on $\DDD$ transitively. 
Thus we have proved  Propositions~\ref{prop:G} and~\ref{prop:geomD}.
\begin{table}
{\small
\begin{eqnarray*}
Q_{{1}}&:& {x}^{2}+{y}^{2}+{z}^{2}=0 \\
Q_{{309}}&:&  \left( 2\,\alpha+2 \right) {x}^{2}+ \left( 3\,\alpha+2
 \right) {y}^{2}+{z}^{2}=0 \\
Q_{{434}}&:&  \left( 3\,\alpha+2 \right) {x}^{2}+ \left( 2\,\alpha+2
 \right) {y}^{2}+{z}^{2}=0 \\
Q_{{1454}}&:& 2\,{z}^{2}+xy=0 \\
Q_{{1535}}&:& 4\,{x}^{2}+4\,{y}^{2}+4\,{z}^{2}+4\,xy+yz+zx=0 \\
Q_{{1628}}&:& {x}^{2}+{y}^{2}+{z}^{2}+xy+yz+zx=0 \\
Q_{{2063}}&:& 3\,{z}^{2}+xy=0 \\
Q_{{2120}}&:& 4\,{x}^{2}+4\,{y}^{2}+4\,{z}^{2}+xy+4\,yz+zx=0 \\
Q_{{2187}}&:& {x}^{2}+{y}^{2}+{z}^{2}+4\,xy+4\,yz+zx=0 \\
Q_{{2445}}&:& 2\,{y}^{2}+zx=0 \\
Q_{{2489}}&:& 3\,{y}^{2}+zx=0 \\
Q_{{2511}}&:&  \left( 2\,\alpha+2 \right) {x}^{2}+{y}^{2}+ \left( 3\,
\alpha+2 \right) {z}^{2}+ \left( 3\,\alpha+2 \right) xy+ \left( 2\,
\alpha+2 \right) yz+zx=0 \\
Q_{{2556}}&:&  \left( 2\,\alpha+3 \right) {x}^{2}+4\,{y}^{2}+ \left( 3\,
\alpha+3 \right) {z}^{2}+ \left( 2\,\alpha+2 \right) xy+ \left( 2\,
\alpha+3 \right) yz+zx=0 \\
Q_{{2592}}&:&  \left( 3\,\alpha+2 \right) {x}^{2}+{y}^{2}+ \left( 2\,
\alpha+2 \right) {z}^{2}+ \left( 3\,\alpha+3 \right) xy+ \left( 2\,
\alpha+3 \right) yz+zx=0 \\
Q_{{2615}}&:&  \left( 3\,\alpha+3 \right) {x}^{2}+4\,{y}^{2}+ \left( 2\,
\alpha+3 \right) {z}^{2}+ \left( 2\,\alpha+3 \right) xy+ \left( 2\,
\alpha+2 \right) yz+zx=0 \\
Q_{{2708}}&:& 2\,{x}^{2}+yz=0 \\
Q_{{2790}}&:& 3\,{x}^{2}+yz=0 \\
Q_{{3082}}&:&  \left( 2\,\alpha+3 \right) {x}^{2}+4\,{y}^{2}+ \left( 3\,
\alpha+3 \right) {z}^{2}+ \left( 3\,\alpha+3 \right) xy+ \left( 3\,
\alpha+2 \right) yz+zx=0 \\
Q_{{3086}}&:&  \left( 3\,\alpha+2 \right) {x}^{2}+{y}^{2}+ \left( 2\,
\alpha+2 \right) {z}^{2}+ \left( 2\,\alpha+2 \right) xy+ \left( 3\,
\alpha+2 \right) yz+zx=0 \\
Q_{{3116}}&:&  \left( 3\,\alpha+3 \right) {x}^{2}+4\,{y}^{2}+ \left( 2\,
\alpha+3 \right) {z}^{2}+ \left( 3\,\alpha+2 \right) xy+ \left( 3\,
\alpha+3 \right) yz+zx=0 \\
Q_{{3122}}&:&  \left( 2\,\alpha+2 \right) {x}^{2}+{y}^{2}+ \left( 3\,
\alpha+2 \right) {z}^{2}+ \left( 2\,\alpha+3 \right) xy+ \left( 3\,
\alpha+3 \right) yz+zx=0 
\end{eqnarray*}
}
\vskip -.1cm
\caption{Vertices of  $D_1$}\label{table:D1}
\end{table}
 \begin{table}
  {\small
 $$
 \begin {array}{ccccccc} -&Q_{{1}}&Q_{{309}}&Q_{{2615}}&Q_{{2511}}&Q_{{3116}}&Q_{{3122}}
\\\noalign{\smallskip}
& -&Q_{{434}}&Q_{{3082}}&Q_{{3086}}&Q_{{2556}}&Q_{{2592}}
\\\noalign{\smallskip}
&& -&Q_{{1535}}&Q_{{1628}}&Q_{{2120}}&Q_{{2187}}
\\\noalign{\smallskip}
&&& -&Q_{{1454}}&Q_{{2489}}&Q_{{2790}}
\\\noalign{\smallskip}
&&&& -&Q_{{2708}}&Q_{{2445}}
\\\noalign{\smallskip}
&&&&& -&Q_{{2063}}
\end {array}
$$
}
\vskip -.1cm
\caption{Adjacency relation on $D_1$}\label{table:adjD1}
\end{table}
\par
%\smallskip
%
Using the matrix $M_0$  and the list $\Dlist$,
we confirm Proposition~\ref{prop:typet} and~\ref{prop:typeT}.
We then calculate a $3150\times 150$ matrix %$\tmat$ 
whose $(\nu, j)$ entry is $t(Q_{\nu}, D_j)$ if $Q_{\nu}\notin D_j$ and $0$ if $Q_{\nu}\in D_j$.
%Then Table~\ref{table:t} is confirmed.
We then calculate a $150\times 150$ matrix $\Tmat$ 
whose $(i, j)$ entry is $T(D_i, D_j)$ if $i\ne j$ and $0$ if $i=j$.
Then Proposition~\ref{prop:symtypeT} is confirmed.
From  the matrix $\Tmat$,  we obtain  the adjacency matrix $A_H$ of  the graph $H$ in Theorem~\ref{thm:mainHfSg}.
It turns out that $H$ has exactly three connected components
whose set of vertices are denoted by 
$\CCC_1, \CCC_2, \CCC_3$.
Then it is easy to confirm that each connected component  of $H$
is a strongly regular graph of parameters $(v, k, \lambda, \mu)=(50, 7, 0, 1)$.
Therefore Theorem~\ref{thm:mainHfSg} is proved. 
Using $\Tmat$ and the sets $\CCC_1, \CCC_2, \CCC_3$,  we confirm 
Proposition~\ref{prop:TDD}.
\par
\smallskip
We then  calculate 
the adjacency matrix $A_{H\sprime}$ of  the graph $H\sprime$
form $\Tmat$.
Then it is easy to confirm that,
for any $i$ and $j$ with $i\ne j$,
 $H\sprime|(\CCC_i\cup \CCC_j)$
is a strongly regular graph of parameters $(v, k, \lambda, \mu)=(100, 22, 0, 6)$.
Thus Theorem~\ref{thm:mainHgSm} is proved.
Using $\Tmat$, the adjacency matrix of   $H$ and the sets $\CCC_1, \CCC_2, \CCC_3$,
we confirm Proposition~\ref{prop:coqliques}.
Then Theorem~\ref{thm:mainMcL} follows from~\cite{MR2917927},
or we can compute the adjacency matrix of $H\spprime$ and confirm Theorem~\ref{thm:mainMcL}
directly.
\par
\bigskip
Note that
the conic $Q_1 \in {\QQQ_5}$ defined by $x^2+y^2+z^2=0$ has a parametric presentation
$$
t\mapsto (t^2+4: 2t: 2\,t^2+2).
$$
Using the list $\EQlist$ and this parametric presentation,
%we see how $Q_1$ and $Q\in {\QQQ_5}$ intersect,
%namely, for each $i>1$,
we can calculate
$$
n(Q_i):=[\nu_1, \nu_2, \nu_3, \nu_4]
$$
for each $i>1$, 
where $\nu_m$ is the number of points in $Q_1\cap Q_i$
at which $Q_1$ and $Q_i$ intersect with intersection multiplicity $m$.
In Table~\ref{table:IP},
we give the number $N$ of conics $Q$
that have an  intersection pattern  with $Q_1$ prescribed by $n(Q)$ and 
$$
a(Q):=|Q_1\cap Q\cap {\Gamma_5}|,\quad 
s(Q):=|\SSS(Q_1)\cap \SSS(Q)|.
$$
\begin{table}
{\small
$$
\renewcommand{\arraystretch}{1.1}
\begin{array}{ccccl}
a(Q) & s(Q) &n(Q) &N &\textrm{an example}\\
\hline
6 & 15 &-&1 &{x}^{2}+{y}^{2}+{z}^{2}=0 \\
0 & 3 &[4,0,0,0]&10& \left( 2\,\alpha+2 \right) {x}^{2}+ \left( 3\,\alpha+2 \right) {y}^{2}+{z}^{2}=0 \\
0 & 3&[0,2,0,0]&20& \left( \alpha+2 \right) {x}^{2}+\alpha\,{y}^{2}+\alpha\,{z}^{2}+yz=0 \\
1 & 5&[0,0,0,1]&24&3\,\alpha\,{x}^{2}+ \left( 3\,\alpha+4 \right) {y}^{2}+ \left( 3\,\alpha+1 \right) {z}^{2}+yz=0 \\
2 & 3&[0,2,0,0]&30& \left( 3\,\alpha+2 \right) {x}^{2}+{y}^{2}+{z}^{2}=0 \\
0 & 0&[0,2,0,0]&30& \left( 2\,\alpha+1 \right) {x}^{2}+ \left( 2\,\alpha+1 \right) {y}^{2}+ \left( 2\,\alpha+1 \right) {z}^{2}+xy+yz+zx=0 \\
2 & 1&[0,2,0,0]&45& \left( 2\,\alpha+3 \right) {x}^{2}+{y}^{2}+{z}^{2}=0 \\
1 & 0&[1,0,1,0]&120&3\,\alpha\,{x}^{2}+ \left( \alpha+2 \right) {y}^{2}+3\,{z}^{2}+3\,xy+ \left( 2\,\alpha+3 \right) yz+zx=0 \\
0 & 1&[4,0,0,0]&390&3\,\alpha\,{x}^{2}+\alpha\,{y}^{2}+\alpha\,{z}^{2}+yz=0 \\
1 & 1&[2,1,0,0]&600& \left( 4\,\alpha+2 \right) {x}^{2}+ \left( 3\,\alpha+4 \right) {y}^{2}+ \left( 3\,\alpha+1 \right) {z}^{2}+yz=0 \\
0 & 0&[4,0,0,0]&1880&4\,{x}^{2}+4\,{y}^{2}+ \left( 2\,\alpha+4 \right) {z}^{2}+ \left( 4\,\alpha+4 \right) xy+yz+zx=0
\end{array}
$$
}
\vskip -.1cm
\caption{Classification of conics by intersection pattern with $Q_1$}\label{table:IP}
\end{table}
\par
%\smallskip
%
Next we prove results in Section~\ref{subsec:6cliques}.
By the adjacency matrix $A_{\graphG}$ of the graph $\graphG$ and the list $\Dlist$,
we can construct the list $\Klist$  of $6$-cliques in  $\graphG$.
Using $\Qlist$, $\Dlist$, $\Klist$ and the adjacency matrix of $H$ obtained from $\Tmat$,
we confirm Propositions~\ref{prop:uniqueK},~\ref{prop:adjHbyK},
and construct the list  $\PK$.
The two $6$-cliques containing $Q_1$  are
\begin{eqnarray*}
K_a &=&\{Q_{{1}},Q_{{309}},Q_{{2511}},Q_{{2615}},Q_{{3116}},Q_{{3122}}\}, \\
K_b &=&\{Q_{{1}},Q_{{434}},Q_{{2556}},Q_{{2592}},Q_{{3082}},Q_{{3086}}\}. 
\end{eqnarray*}
Then their partners in $\PK$ are
\begin{eqnarray*}
K_a\sprime &=&\{Q_{{8}},Q_{{171}},Q_{{827}},Q_{{936}},Q_{{1973}},Q_{{2038}}\}, \\
K_b\sprime &=&\{Q_{{22}},Q_{{160}},Q_{{816}},Q_{{947}},Q_{{1984}},Q_{{2034}}\},
\end{eqnarray*}
the defining equations of whose  members are given in Table~\ref{table:Ksprime}.
\begin{table}
{\small
\begin{eqnarray*}
Q_{{8}} &: &
3\,\alpha\,{x}^{2}+ \left( 3\,\alpha+4 \right) {y}^{2}+ \left( 3\,
\alpha+1 \right) {z}^{2}+yz=0 \\
Q_{{171}} &: &
2\,\alpha\,{x}^{2}+ \left( 2\,\alpha+1 \right) {y}^{2}+ \left( 2\,\alpha+4 \right) {z}^{2}+yz=0 \\
Q_{{827}} &: &
 \left( 2\,\alpha+4 \right) {x}^{2}+2\,\alpha\,{y}^{2}+ \left( 2\,\alpha+1 \right) {z}^{2}+zx=0 \\
Q_{{936}} &: &
 \left( 3\,\alpha+1 \right) {x}^{2}+3\,\alpha\,{y}^{2}+ \left( 3\,\alpha+4 \right) {z}^{2}+zx=0 \\
Q_{{1973}} &: &
 \left( 3\,\alpha+4 \right) {x}^{2}+ \left( 3\,\alpha+1 \right) {y}^{2}+3\,\alpha\,{z}^{2}+xy=0 \\
Q_{{2038}} &: &
 \left( 2\,\alpha+1 \right) {x}^{2}+ \left( 2\,\alpha+4 \right) {y}^{2}+2\,\alpha\,{z}^{2}+xy=0 \\
Q_{{22}} &: &
2\,\alpha\,{x}^{2}+ \left( 2\,\alpha+4 \right) {y}^{2}+ \left( 2\,\alpha+1 \right) {z}^{2}+yz=0 \\
Q_{{160}} &: &
3\,\alpha\,{x}^{2}+ \left( 3\,\alpha+1 \right) {y}^{2}+ \left( 3\,\alpha+4 \right) {z}^{2}+yz=0 \\
Q_{{816}} &: &
 \left( 3\,\alpha+4 \right) {x}^{2}+3\,\alpha\,{y}^{2}+ \left( 3\,\alpha+1 \right) {z}^{2}+zx=0 \\
Q_{{947}} &: &
 \left( 2\,\alpha+1 \right) {x}^{2}+2\,\alpha\,{y}^{2}+ \left( 2\,\alpha+4 \right) {z}^{2}+zx=0 \\
Q_{{1984}} &: &
 \left( 2\,\alpha+4 \right) {x}^{2}+ \left( 2\,\alpha+1 \right) {y}^{2}+2\,\alpha\,{z}^{2}+xy=0 \\
Q_{{2034}} &: &
 \left( 3\,\alpha+1 \right) {x}^{2}+ \left( 3\,\alpha+4 \right) {y}^{2}+3\,\alpha\,{z}^{2}+xy=0.
\end{eqnarray*}
}
\vskip -.2cm
\caption{Conics in $K_a\sprime$ and $K_b\sprime$}\label{table:Ksprime}
\end{table}
Each of 
these conics
intersects $Q_1$ only  at one point with intersection multiplicity $4$.
For example, $Q_8$ intersects $Q_1$ only at $(0:3:1)$.
Since $\Aut({\Gamma_5})$ acts on ${\QQQ_5}$ transitively, 
Proposition~\ref{prop:geomK} is proved.
\begin{remark}
From~Table~\ref{table:IP}, we see that 
there exist exactly $12$ conics $Q\sprime\in {\QQQ_5}$
that intersect $Q_1$ only  at one point with intersection multiplicity $4$
but are not contained in $K_a\sprime \cup K_b\sprime$.
An example of such a conic is 
$$
4\,\alpha\,{x}^{2}+ \left( 4\,\alpha+4 \right) {y}^{2}+ \left( 4\,
\alpha+1 \right) {z}^{2}+yz=0.
$$
These $12$  conics are contained in $\wt{\CCC}_1$.
\end{remark}
\par
\smallskip
The proofs in Sections~\ref{subsec:SQ} and~\ref{subsec:RRRTTT} 
are analogous  and we omit the details.
\begin{remark}
Some families of strongly regular graphs  have been  constructed from Hermitian varieties
in Chakravarti~\cite{MR1210086}.
\end{remark}
\section{Group-theoretic construction}\label{sec:proof2}
In order to verify the group-theoretic construction (Theorems~\ref{thm:S5},~\ref{thm:T7},~\ref{thm:A7} and Proposition~\ref{prop:twobb21}),
we make the following computational data.
Since the order $378000$ of $\PGU_3(\F_{25})$ is large,
it uses too much memory  to make the list of all elements of $\PGU_3(\F_{25})$.
Instead
we make the lists of elements of 
\begin{eqnarray*}
G_S &:=& \stab(p_1), \quad \textrm{where $p_1=(0:1:2)\in \PPP_5$, \quad and}\\
G_T &:=&\textrm{a complete set of representatives of $\PGU_3(\F_{25})/G_S$}.
\end{eqnarray*}
Then we have $|G_S|=3000$ and 
$|G_T|=126$,
and each element of $\PGU_3(\F_{25})$ is uniquely written as 
$\tau\sigma$, where $\sigma\in G_S$ and $\tau\in G_T$. 
We then calculate the permutation on the set $\PPP_5$
induced by each of the  $3000+126$ elements of $G_S$ and $G_T$.
From this list of permutations,
we calculate the permutation on the set $\QQQ_5$
induced by each   element of $G_S$ and $ G_T$,
and from this list,
we calculate the permutation on the set $\DDD$
induced by each   element of $G_S$ and $G_T$.
Thus we obtain three permutation representations of $\PGU_3(\F_{25})$
on $\PPP_5$, $\QQQ_5$ and $\DDD$,
each of which is faithful.
\begin{remark}
In order to determine the structure of subgroups of $\PGU_3(\F_{25})$,
it is more convenient to use these permutation representations 
than to handle  $3\times 3$ matrices with components in $\F_{25}$.
\end{remark}
Then we calculate $\stab(Q_1)$ and its subgroups 
$$
\stab(Q_1, Q):=\stab(Q_1)\cap \stab(Q)
$$
for each  $Q\in \QQQ_5$.
Combining this data with the adjacency matrix $A_{\graphG}$ of the graph $\graphG$,
we confirm the first-half of Theorem~\ref{thm:S5}.
\par
\smallskip
If $Q\in \QQQ_5$ is distinct from $Q_1$,
then $\stab(Q_1, Q)$ is isomorphic to one of the following groups:
$$
0, \;\; \Z/2\Z, \;\; \Z/3\Z, \;\; (\Z/2\Z)^2, \;\; \DDDD_{8}, \;\; \DDDD_{10}, \;\; \DDDD_{12}, \;\; \AAAA_4,
$$
where $\DDDD_{2n}$ is the dihedral group of order $2n$.
Table~\ref{table:IP} is refined to Table~\ref{table:orbQs} by using this new data.
The  action of $\stab(Q_1)\cong \SSSS_5$ decomposes $\QQQ_5$ into $64$  orbits,
and each row of Table~\ref{table:orbQs}
contains
$$
 N\cdot |\stab(Q_1, Q)|/120
$$
orbits of size $120/|\stab(Q_1, Q)|$.
\begin{table}
{\small
$$
\renewcommand{\arraystretch}{1.1}
\begin{array}{cccccc}
\rm{No.} &a(Q) & s(Q) &n(Q) &\stab(Q_1, Q)&N \\
\hline
1& 6 & 15 &-&  \SSSS_5& 1\\ 
2& 0 & 3 &[4, 0, 0, 0] &  \AAAA_4& 10  \\ 
3& 0 & 3 &[0, 2, 0, 0] &  \DDDD_{12}& 20  \\ 
4& 1 & 5 &[0, 0, 0, 1] &  \DDDD_{10}& 24  \\ 
5& 2 & 3 &[0, 2, 0, 0] &  \DDDD_{8}& 30  \\ 
6& 0 & 0 &[0, 2, 0, 0] & \DDDD_{12}& 30 \\ 
7& 2 & 1 &[0, 2, 0, 0] &  \DDDD_8& 45 \\ 
8& 1 & 0 &[1, 0, 1, 0] &  0& 120 \\ 
9& 0 & 1 &[4, 0, 0, 0] &  \Z/2\Z& 180 \\ 
10& 0 & 1 &[4, 0, 0, 0] &  (\Z/2\Z)^2& 210 \\ 
11& 1 & 1 &[2, 1, 0, 0] &  \Z/2\Z& 600 \\ 
12& 0 & 0 &[4, 0, 0, 0] &  0& 720 \\ 
13& 0 & 0 &[4, 0, 0, 0] &  \Z/2\Z& 900 \\ 
14& 0 & 0 &[4, 0, 0, 0] &  \Z/3\Z& 80 \\ 
15& 0 & 0 &[4, 0, 0, 0] &  (\Z/2\Z)^2& 180
\end{array}
$$
}
\vskip -.1cm
\caption{$\stab(Q_1, Q)$}\label{table:orbQs}
\end{table}
From each of these orbits other than $\{Q_1\}$,
we choose a representative conic $Q$ and confirm the following:
$$
\renewcommand{\arraystretch}{1.2}
\begin{array}{lll}
%\textrm{$Q=Q_1$} &\Longleftrightarrow& |\gen{\stab(Q_1), \stab(Q)}|=120, \\
Q\in D_1 &\Longrightarrow & |\,\gen{\stab(Q_1), \stab(Q)}\,|=2520,\\
Q\notin D_1 &\Longrightarrow& |\,\gen{\stab(Q_1), \stab(Q)}\,|>2520.  \\
\end{array}
$$
Thus the second-half of Theorem~\ref{thm:S5}
is verified.
\begin{remark}
The conics in $\QQQ_5$ that are in the connected component $D_1$ of $\graphG$ 
but are not adjacent to $Q_1$ in $\graphG$ form one of the three orbits of size $10$
in the $6$th row of Table~\ref{table:orbQs}.
\end{remark}
We then calculate 
$\stab(D_1)$ and its subgroups 
$$
\stab(D_1, D):=\stab(D_1)\cap \stab(D)
$$
for each  $D\in \DDD$.
An isomorphism $\kappa: D_1\isom\, T(7)$ of graphs is obtained from the triangle in Table~\ref{table:adjD1}
by putting  $\kappa(Q_{\nu})=\{i, j\}$ if $Q_{\nu}\in D_1$ is at the $i$th row and the $j$th column of the triangle.
This map $\kappa$ gives rise to  a homomorphism
$$
\stab(D_1) \to\Aut(T(7))=\SSSS_7.
$$
We confirm that this homomorphism is injective, and 
verify Theorem~\ref{thm:T7}.
 Using the matrix $\Tmat$, 
 we also confirm Theorem~\ref{thm:A7} and Proposition~\ref{prop:twobb21}.
\begin{remark}
Each  connected component of the graph $H$ is an orbit of
the action of the subgroup $\PSU_{3} (\F_{25})$ of index $3$ in $\PGU_3(\F_{25})$
on $\DDD$.
\end{remark}
Combining all the results above,
we can construct the Hoffmann-Singleton graph and the Higman-Sims  graph  from $\PGU_3(\F_{25})$ 
without using any geometry.
We put
$$
\Delta:=\PGU_{3}(\F_{25}) \cap \PGO_{3}(\F_{25}).
$$
\begin{remark}
We have $\PGO_{3}(\F_{25})=\shortset{M\in \PGL_3(\F_{25})}{\textrm{${}^t M M$ is a diagonal matrix}}$.
Since $Q_1$ is defined by $x^2+y^2+z^2=0$,
we have $\Delta=\stab(Q_1)$.
\end{remark}
Consider the following five elements of $\PGU_3(\F_{25})$:
$$
\begin{array}{lll}
g_2:=\left[ \begin {array}{ccc} 1&0&0\\\noalign{\medskip}
0&\omega&0\\\noalign{\medskip}
0&0&\omega\inv\end {array} \right] &
(\omega:=2+3\alpha), &
%g_2:=\left[ \begin {array}{ccc} 1&0&0\\\noalign{\medskip}0&2+3\alpha&0
%\\\noalign{\medskip}0&0&2+2\alpha\end {array} \right], \quad &
g_3:=\left[ \begin {array}{ccc} 1&3&2\\\noalign{\medskip}3&1&2
\\\noalign{\medskip}2&2&1\end {array} \right], \quad   \\
g_4:=\left[ \begin {array}{ccc} 0&1&0\\\noalign{\medskip}1&0&0
\\\noalign{\medskip}0&0&4\end {array} \right], &
g_5:= \left[ \begin {array}{ccc} 1&0&0\\\noalign{\medskip}0&0&4
\\\noalign{\medskip}0&4&0\end {array} \right], &
g_6:= \left[ \begin {array}{ccc} 0&1&0\\\noalign{\medskip}1&0&0
\\\noalign{\medskip}0&0&1\end {array} \right]. \\
\end{array}
$$
\begin{remark}
The elements $g_i$ ($i=2, \dots, 6$) belong to $\stab(D_1)$,
and correspond to $(1, 2)(i, i+1)\in \AAAA_7$
by the isomorphism $\stab(D_1)\cong \AAAA_7$ induced  by $\kappa: D_1\cong T(7)$.
Since these five permutations $(1, 2)(i, i+1)$  generate $\AAAA_7$,
we see that the elements $g_2, \dots, g_6$
generate $\stab(D_1)$.
\end{remark}
\begin{theorem}\label{thm:purelygroup}
Let $\Gamma$ be the subgroup of $\PGU_3(\F_{25})$ generated by $g_2, \dots, g_6$ above.
Then $\Gamma$ is isomorphic to $\AAAA_7$.
Moreover, if $\gamma\in \PGU_{3}(\F_{25})$ satisfies $\Delta\,\cap\, \gamma\inv \Delta \gamma\cong \AAAA_4$,
then $\Gamma$ is generated by $\Delta$ and $\gamma\inv \Delta \gamma$.
\par
Let $\VVV$ denote the set of subgroups of $\PGU_3(\F_{25})$ conjugate to $\Gamma$.
Then, we have $|\VVV|=150$, and  for  distinct elements $\Gamma\sprime, \Gamma\spprime\in \VVV$,
the group $\Gamma\sprime\cap \Gamma\spprime$ is isomorphic to one of the following:
$$
\AAAA_6, \quad  \PSL_2(\F_7),  \quad (\AAAA_4\times 3):2,  \quad \AAAA_5.
$$
For $n=360, 168, 72, 60$, we define  a subset $\EEE_{n}$ of the set of unordered pairs of 
distinct elements of $\VVV$ by 
$$
\EEE_{n} := \set{ \{\Gamma\sprime, \Gamma\spprime\} }{|\Gamma\sprime\cap \Gamma\spprime|=n}.
$$
Then the graph $(\VVV, \EEE_{360})$ has exactly three connected components $\CCC_1, \CCC_2, \CCC_3$,
and each $\CCC_i$ is the  Hoffmann-Singleton graph.
Moreover, for any $i\ne j$,
the graph $(\VVV, \EEE_{360}\cup \EEE_{168})|{(\CCC_i\cup \CCC_j)}$ is the Higman-Sims  graph.
\end{theorem}
We can also construct the McLaughlin graph from $\VVV$ and $\EEE_n$  by
the recipe of Inoue~\cite{MR2917927} in the same way as Theorem~\ref{thm:mainMcL}.
\bibliographystyle{plain}

\def\cftil#1{\ifmmode\setbox7\hbox{$\accent"5E#1$}\else
  \setbox7\hbox{\accent"5E#1}\penalty 10000\relax\fi\raise 1\ht7
  \hbox{\lower1.15ex\hbox to 1\wd7{\hss\accent"7E\hss}}\penalty 10000
  \hskip-1\wd7\penalty 10000\box7} \def\cprime{$'$} \def\cprime{$'$}
  \def\cprime{$'$} \def\cprime{$'$}

\end{document}